\pdfoutput=1
\RequirePackage{ifpdf}
\ifpdf 
\documentclass[pdftex]{sigma}
\else
\documentclass{sigma}
\fi

\usepackage[all]{xy}
\usepackage{mathrsfs}
\usepackage{tikz}

\numberwithin{equation}{section}

\newtheorem{Theorem}{Theorem}[section]
\newtheorem*{Theorem*}{Theorem}

\newtheorem{lem}[Theorem]{Lemma}
\newtheorem{prop}[Theorem]{Proposition}
 { \theoremstyle{definition}
\newtheorem{defn}[Theorem]{Definition}

\newtheorem{eg}[Theorem]{Example}
\newtheorem{Remark}[Theorem]{Remark} }

\newcommand{\e}{\varepsilon}

\newcommand{\usslevi}{U_q(\mathfrak{l}^{\,\mathrm{s}}_S)}
\newcommand\del{\partial}
\newcommand\adel{\overline{\partial}}
\newcommand{\qphs}{\mathcal{O}_q\big(G/L^{\mathrm{s}}_S\big)}
\newcommand\unit{\mathrm{U}}

\begin{document}
\allowdisplaybreaks

\newcommand{\arXivNumber}{2202.09842}

\renewcommand{\PaperNumber}{070}

\FirstPageHeading

\ShortArticleName{Bimodule Connections for Relative Line Modules}

\ArticleName{Bimodule Connections for Relative Line Modules\\ over the Irreducible Quantum Flag Manifolds}

\Author{Alessandro CAROTENUTO and R\'eamonn \'O~BUACHALLA}
\AuthorNameForHeading{A.~Carotenuto and R.~\'O~Buachalla}
\Address{Mathematical Institute of Charles University, Sokolovsk\'a 83, Prague, Czech Republic}
\Email{\href{mailto:acaroten91@gmail.com}{acaroten91@gmail.com}, \href{mailto:obuachalla@karlin.mff.cuni.cz}{obuachalla@karlin.mff.cuni.cz}}

\ArticleDates{Received April 14, 2022, in final form September 18, 2022; Published online October 01, 2022}

\Abstract{It was recently shown (by the second author and D\'iaz Garc\'ia, Krutov, Somberg, and Strung) that every relative line module over an irreducible quantum flag manifold $\mathcal{O}_q(G/L_S)$ admits a unique $\mathcal{O}_q(G)$-covariant connection with respect to the Heckenberger--Kolb differential calculus $\Omega^1_q(G/L_S)$. In this paper we show that these connections are bimodule connections with an invertible associated bimodule map. This is proved by applying general results of Beggs and Majid, on principal connections for quantum principal bundles, to the quantum principal bundle presentation of the Heckenberger--Kolb calculi recently constructed by the authors and D\'iaz Garc\'ia. Explicit presentations of the associated bimodule maps are given first in terms of generalised quantum determinants, then in terms of the FRT presentation of the algebra $\mathcal{O}_q(G)$, and finally in terms of Takeuchi's categorical equivalence for relative Hopf modules.}


\Classification{46L87; 81R60; 81R50; 17B37; 16T05}

\section{Introduction}

Left connections for left modules (and right connections for right modules) are basic structures in noncommutative geometry. They generalise connections for a vector bundle $V$, where of course the left and right actions on the space of sections $\Gamma(V)$ coincide. However, in the noncommutative setting it is most convenient, where possible, to consider bimodules. This prompted the development of a theory of bimodule connections, which is to say, left (right) connections, together with a compatibility for the right (left) action, described in terms of an associated bimodule map. The literature on bimodule connections spans twenty-five years, from the original works \cite{MdV01, DVMa96, DVMi96, Madore99, MMM95} to more recent works such as \cite{AS14,BM17,BGL20, Aryan2020,M20}. The interested reader can find a very thorough exposition of these topics in the recent monograph \cite{BeggsMajid:Leabh}.

Quantum principal bundles and \emph{principal} connections were introduced by Brzezi\'{n}ski and Majid \cite{TBSM1} as a method for constructing left connections for modules associated to right comodule algebras $B = P^{\textrm{co}(H)}$. Now associated modules are naturally bimodules, and so, it is natural to ask when a principal connection induces a bimodule connection on an associated module. This question has been thoroughly investigated by Beggs and Majid \cite[Section~5.4]{BeggsMajid:Leabh}, resulting in a~number of sufficient criteria for determining when bimodule connections are produced.

This provides us with powerful machinery to construct bimodule connections, and this paper concerns itself with an application of this machinery to the theory of Drinfeld--Jimbo quantum groups.
Examples of bimodule connections in the setting of Drinfeld--Jimbo quantum groups have been constructed for the quantum plane in \cite{MouradDVMadore94}, then for the relative line modules $\mathcal{E}_{k}$ over the standard Podle\'s sphere in \cite{KLVSCP1} and \cite[Example 5.23]{BeggsMajid:Leabh}, for the quantum projective plane in \cite{KM11}, and for the first-order Heckenberger--Kolb calculi of quantum projective space in \cite{M20}. In \cite{HVBQFM} the second author, together with D\'iaz Garc\'ia, Krutov, Somberg and Strung, showed that the relative line modules over all irreducible quantum flag manifolds $\mathcal{O}_q(G/L_S),$ with the first order Heckenberger--Kolb calculus, admit a unique $\mathcal{O}_q(G)$-covariant connection. The same connection was then used when proving the Borel Weil theorem for quantum Grassmannians in \cite{BwGrass} and for the irreducible quantum flag manifolds in \cite{CDOBBW}.
In this paper we use the theory of quantum principal bundles to extend Beggs and Majid's work on the standard Podle\'{s} sphere to the setting of the irreducible quantum flag manifolds $\mathcal{O}_q(G/L_S)$, and so, produce a large family of systematically constructed bimodule connections for Drinfeld--Jimbo quantum spaces.

In more detail, the irreducible quantum flag manifolds are a broad family of quantum homogeneous spaces, with the standard Podle\'s sphere, and more generally quantum projective space, as the most tractable examples. From the seminal work of Heckenberger and Kolb \cite{HK,HKdR}, we know that each $\mathcal{O}_q(G/L_S)$ possesses an essentially unique $\mathcal{O}_q(G)$-covariant \mbox{$q$-deformation} of its classical de Rham complex $\Omega^{\bullet}_q(G/L_S)$, constituting one of the most important structures in the noncommutative geometry of quantum groups. The relative line modules $\mathcal{E}_k$ over each $\mathcal{O}_q(G/L_S)$, which are indexed by integers $k \in \mathbb{Z}$, were shown in \cite{HVBQFM} to possesses a unique left $\mathcal{O}_q(G)$-covariant connection $\nabla$ with respect to $\Omega^{\bullet}_q(G/L_S)$. Now as established in the series of papers \cite{CDOBBW,BwGrass,GAPP}, the Heckenberger--Kolb calculi admit quantum principal bundle descriptions. In this presentation, the total algebra of the quantum bundle is the quantum Poisson homogeneous space\looseness=-1
\begin{gather*}
\mathcal{O}_q\big(G/L^{\,\mathrm{s}}_S\big) \simeq \bigoplus_{k \in \mathbb{Z}} \mathcal{E}_k,
\end{gather*}
and the first-order calculus is the Heckenberger--Kolb calculus over $\mathcal{O}_q\big(G/L^{\,\mathrm{s}}_S\big)$ introduced in \cite[Section~3.2.3]{HKdR}. The zero map was then shown to be a strong principal connection, allowing the connection for each line module $\mathcal{E}_k$ to be realised as an associated connection. Combining this result with Beggs and Majid's general results on bimodule connections allows us to conclude that each $\nabla$ is a bimodule connection, as are its holomorphic and anti-holomorphic summands. An immediate application is a description, for each $k \in \mathbb{Z}_{>0}$, of the connections on $\mathcal{E}_{\pm k}$ as the $k$-fold tensor product of the $\mathcal{E}_{\pm 1}$ connection.

Our proof of the bimodule property has the advantage of being simple and conceptual, but the disadvantage of not providing an explicit description of the associated bimodule map $\sigma$. To address this we produce three complementary explicit descriptions of $\sigma$. The first approach uses a family of identities, which we loosely term \emph{generalised determinants}, coming from a dual basis description of the copy of the trivial module in the tensor product $V_{\varpi_x} \otimes {}^{\vee}V_{\varpi_x}$, where $V_{\varpi_x}$ is the simple $U_q(\mathfrak{g})$-module corresponding to the crossed Dynkin node, and ${}^{\vee}V_{\varpi_x}$ is the dual module. This extends the foundational work done in \cite{Maj} for the special case of standard Podl\'es sphere.
The second approach realises the generalised determinant identities, for the $A$ and $B$-series, in terms of the well-known FRT presentation \cite{FRT89} of the non-exceptional quantum coordinate algebras $\mathcal{O}_q(G)$. For the $A$-series, which is to say for the quantum Grassmannians, this gives an expression of the bimodule map in terms of the well-known quantum minors. Finally, we use Takeuchi's equivalence for relative Hopf modules \cite{Tak} to give an alternative description of $\sigma$ for all the irreducible quantum flag manifolds.

We consider our results on tensor products to be an important step towards a conceptual explanation for the phenomenon of $q$-deformed additivity of tensor product curvature \cite[Example~5.51]{BeggsMajid:Leabh}, and we intend to investigate this in future work.

\subsection{Summary of the paper}

The paper is organised as follows: in Section~\ref{section2} we recall basic material on first-order differential calculi, Hopf--Galois extensions, and quantum principal bundles. Moreover, we briefly recall necessary material about Drinfeld--Jimbo quantum groups, their quantum flag manifolds, and their relative line modules.

In Section~\ref{section3} we give a simple accompaniment to Beggs and Majid's criteria for a strong principal connection to induce bimodule connections. We then recall the recent classification of $\mathcal{O}_q(G)$-covariant connections for line modules over the irreducible quantum flag manifolds~\cite{HVBQFM}. We also recall the recent quantum principal bundle description of the Heckenberger--Kolb calculi~\cite{CDOBBW}. Building directly on this material, we prove the main result of the paper:

\begin{Theorem}
For each line module $\mathcal{E}_k$, the connection $\nabla$ is a bimodule connection. Moreover, the associated bimodule map is a left $\mathcal{O}_q(G)$-comodule isomorphism.
\end{Theorem}

We also make some comments on the associated tensor product connections.

In Section~\ref{section4} explicit presentations of the associated bimodule maps are given in terms of generalised quantum determinants, then in terms of the FRT presentations of $\mathcal{O}_q(\mathrm{SU}_n)$ and $\mathcal{O}_q(\mathrm{SO}_{2n+1})$, and finally in terms of Takeuchi's categorical equivalence for relative Hopf modules.

In Appendix~\ref{appendixA} we recall some useful material about the irreducible quantum flag manifolds.

\section{Preliminaries}\label{section2}

In this section we briefly recall standard material on first-order differential calculi, Hopf--Galois extensions, and quantum principal bundles. We refer the reader to \cite{BeggsMajid:Leabh} for a more detailed presentation. Moreover, we briefly recall necessary material about Drinfeld--Jimbo quantum groups, their quantum flag manifolds, and their relative line modules.

\subsection{First-order differential calculi and bimodule connections}\label{section2.1}

In this subsection we recall the notion of a bimodule connection for a differential calculus, see \cite[Section~3.4]{BeggsMajid:Leabh} for further details. A {\em first-order differential calculus} (fodc) over an algebra $B$ is a pair $\big(\Omega^1(B),\mathrm{d}\big)$, where $\Omega^1(B)$ is a $B$-bimodule and $\mathrm{d}\colon B \to \Omega^1(B)$ is a derivation such that~$\Omega^1$ is generated as a left $B$-module by those elements of the form~$\mathrm{d} b$, for~$b \in B$.
The {\em universal fodc} over $B$ is the pair $\big(\Omega^1_u(B), \mathrm{d}_u\big)$, where $\Omega^1_u(B)$ is the kernel of the multiplication map $m_B\colon B \otimes B \to B$, and $\mathrm{d}_u\colon B \to \Omega^1_u(B)$ is defined by $\mathrm{d}_u b = 1 \otimes b - b \otimes 1$. Every fodc over~$B$ can be obtained as the quotient of $\Omega^1_u(B)$ by a $B$-sub-bimodule $N\subset \Omega^1_u(B)$. For any subalgebra $B' \subseteq B$, the \emph{restriction} of a~fodc over $B$ to $B'$ is the fodc $\Omega^1(B') \subseteq \Omega^1(B)$ over $B'$ generated by the elements $\mathrm{d}b$, for~$b \in B'$.

For $\Omega^1(B)$ a fodc over an algebra $B$ and $\mathcal{F}$ a left $B$-module, a \emph{left connection} on $\mathcal{F}$ is a~$\mathbb{C}$-linear map $\nabla\colon \mathcal{F} \to \Omega^1(B)\otimes_B \mathcal{F}$ satisfying the \emph{left Leibniz rule}
\begin{gather*}
\nabla(bf) = \mathrm{d} b \otimes f + b \nabla f, \qquad \textrm{for all } \quad b \in B, f \in \mathcal{F}.
\end{gather*}
A \emph{left bimodule connection} on $\mathcal{F}$ is a pair $(\nabla,\sigma)$ where $\nabla$ is a left connection and $\sigma\colon \mathcal{F} \otimes_B \Omega^1(B) \rightarrow \Omega^1(B) \otimes_B \mathcal{F} $ is a bimodule map satisfying
\begin{gather}\label{eqn:BiC}
 \sigma (f\otimes \mathrm{d}b)= \nabla (fb)-\nabla(f)b.
\end{gather}
Note that in the commutative case $\sigma(f \otimes \mathrm{d} b) = \mathrm{d} b \otimes f$, that is, $\sigma$ reduces to the flip map.

Let $\mathcal{F}$ be a bimodule, $(\nabla_{\mathcal{F}}, \sigma_{\mathcal{F}})$ a bimodule connection, and $\mathcal{E}$ a bimodule with bimodule connection $(\nabla_\mathcal{E}, \sigma_{\mathcal{E}})$, we define the connection for $\mathcal{F} \otimes_B \mathcal{E}$ as follows
\begin{gather}\label{eqn:bimcontenpr}
 \nabla_{\mathcal{F}\otimes \mathcal{E}}= \nabla_{\mathcal{F}}\otimes \mathrm{id}_{\mathcal{E}}+(\sigma_{\mathcal{F}}\otimes \mathrm{\mathrm{id}}_\mathcal{E})\circ (\mathrm{id}_{\mathcal{F}}\otimes \nabla_{\mathcal{E}}).
\end{gather}
Note that $\nabla_{\mathcal{F} \otimes \mathcal{E}}$ is again a bimodule connection, with bimodule map $(\sigma_{\mathcal{F}} \otimes \mathrm{id}) \circ (\mathrm{id} \otimes \sigma_{\mathcal{E}})$.

\begin{eg} \label{eg:trivialBimodConn}
For any fodc $\big(\Omega^1(B),\mathrm{d}\big)$, through the isomorphism $\Omega^1(B) \simeq \Omega^1(B) \otimes_B B$, we can consider $\mathrm{d}\colon B \to \Omega^1(B)$ as a connection for $B$. Since
\[
\mathrm{d}(bb') - \mathrm{d}(b)b' = b\mathrm{d}(b'),
\]
it holds that this is a bimodule connection with associated bimodule map
\begin{gather*}
\sigma(b \otimes \mathrm{d} b') = b\mathrm{d} b' \otimes 1.
\end{gather*}
This elementary bimodule connection will be used in the proof of Lemma~\ref{prop:TAKtheta} below.
\end{eg}

Let $\mathcal{F}$ and $\mathcal{F}'$ be two left $B$-bimodules, and let $\nabla\colon \mathcal{F} \to \Omega^1(B) \otimes_B \mathcal{F} $ and $\nabla'\colon \mathcal{F}' \to \Omega^1(B) \otimes_B \mathcal{F}'$ be two connections, then we say that $\nabla$ and $\nabla'$ are \emph{isomorphic} if we have an isomorphism of $B$-bimodules $\alpha\colon \mathcal{F} \to \mathcal{F}'$ such that the following diagram commutes:
\[
\xymatrix{
\mathcal{F} \ar[d]_{\alpha} \ar[rrr]^{\nabla} & & & \Omega^1(B) \otimes_B \mathcal{F} \ar[d]^{\mathrm{id} \otimes \alpha} \\
\mathcal{F}' \,\, \ar[rrr]_{\nabla'} & & & \Omega^1(B) \otimes_B \mathcal{F}'.
}
\]

\subsection{Quantum principal bundles}\label{section2.2}

Let $(P,\Delta_R)$ be a right $H$-comodule algebra, for some Hopf algebra $H$. A fodc over $P$ is said to be {\em right $H$-covariant} if the following (necessarily unique) map is well defined
\begin{gather*}
\Delta_R\colon \ \Omega^1(P) \to \Omega^1(P) \otimes H, \qquad p \mathrm{d} q \mapsto p_{(0)} \mathrm{d} q_{(0)} \otimes p_{(1)} q_{(1)}.
\end{gather*}
Covariance for a fodc over a left comodule algebra is defined similarly.

Let $(P,\Delta_R)$ be a right $H$-comodule algebra. We say that $P$ is a Hopf--Galois extension of $B = P^{\,\mathrm{co}(H)}$ a bijection is given by
\[
\mathrm{can} := (m_P \otimes \mathrm{id}) \circ (\mathrm{id} \otimes \Delta_R)\colon \ P \otimes_B P \to P \otimes H,
\]
where $m_P$ denotes the multiplication of $P$. A \emph{right $H$-principal comodule algebra} is a right $H$-comodule algebra $P$, such that $P$ is a Hopf--Galois extension of $B$, and such that $P$ is faithfully flat as a $B$-bimodule.

\begin{defn} \label{qpb}
Let $H$ be a Hopf algebra. A {\em quantum principal $H$-bundle} is a pair $\big(P,\Omega^1(P)\big)$, consisting of a right $H$-principal comodule algebra $(P,\Delta_R)$, and a right $H$-covariant calculus~$\Omega^1(P)$, such that, if $N \subseteq \Omega^1_u(P)$ is the sub-bimodule of the universal calculus corresponding to $\Omega^1(P)$, we have $\mathrm{ver}(N) = P \otimes I$, for some $\mathrm{Ad}$-subcomodule right ideal
\[
I \subseteq H^+ := \ker(\e\colon H \to \mathbb{C}),
\]
where
$\mathrm{Ad} \colon H \to H \otimes H$ is defined by $\mathrm{Ad}(h) := h_{(2)} \otimes S(h_{(1)}) h_{(3)}$.
A quantum principal bundle is called \textit{strong} when it holds that
$\Omega^1(B)P=P\Omega^1(B)$.
\end{defn}
We can now define a principal connection for a quantum principal bundle. In Section~\ref{subsection:ConnFromPrinCs} below we will show how to define connections for a certain class of noncommuative vector bundles starting from a principal connection.

\begin{defn}
A {\em principal connection} for a quantum principal $H$-bundle $\big(P,\Omega^1(P)\big)$ is a left $P$-module, right $H$-comodule, projection $\Pi\colon \Omega^1(P) \to \Omega^1(P)$ satisfying
\[
\ker(\Pi) = P\Omega^1(B)P.
\]
\end{defn}
A principal connection $\Pi$ is called {\em strong} if $(\mathrm{id} - \Pi) \big(\mathrm{d} P\big) \subseteq \Omega^1(B)P$. It is clear that for any strong principal bundle, all principal connections are strong.

\subsection{Connections from principal connections} \label{subsection:ConnFromPrinCs}

For $P$ a right principal $H$-comodule algebra, we denote by ${}_P\mathrm{Mod}$ the category of left $P$-modules, and we define a functor
\begin{gather*}
\Psi\colon \ {}^{H}\mathrm{Mod} \to {}_B\mathrm{Mod}_B, \qquad V \mapsto P \, \square_H V,
\end{gather*}
where the right $B$-action is given by multiplication on the first tensor factor, and that acts on morphisms $\gamma\colon V \to W$ as $\Psi(\gamma) = \mathrm{id} \otimes \gamma$. We call any left $P$-module of the form $\Psi(V)$, for some $V \in {}^H\mathrm{Mod}$, an \emph{associated module}. We can use principal connections to define a connection $\nabla\colon \mathcal{F} \to \Omega^1(B) \otimes_B \mathcal{F}$ for any left $P$-module of the form $\mathcal{F} := \Psi(V).$
Note first that we have a~natural embedding
\begin{gather*}
j\colon \ \Omega^1(B) \otimes_B \mathcal{F} \hookrightarrow \Omega^1(B)P \, \square_H V
\end{gather*}
given by the multiplication map.
A strong principal connection $\Pi$ determines a connection $\nabla$ on $\mathcal{F}$ by
\begin{gather*}
 \nabla\colon \ \mathcal{F} \to \Omega^1(B) \otimes_B \mathcal{F}, \qquad \sum_i p_i \otimes v_i \mapsto j^{-1} \big((\mathrm{id} - \Pi)(\mathrm{d} p_i) \otimes v_i \big).
\end{gather*}
Indeed, since $\mathrm{d}$ and the projection $\Pi$ are both right \mbox{$H$-comodule} maps, their composition $(\mathrm{id} - \Pi) \allowbreak \circ \mathrm{d} \otimes \mathrm{id}$ is a right $H$-comodule map. Hence the image of $(\mathrm{id} - \Pi) \circ \mathrm{d}\otimes \mathrm{id}_V$ is contained in $j\big(\Omega^1(B) \otimes_B \mathcal{F}\big)$.

Let us now additionally assume that $P$ admits a left $A$-coaction giving it the structure of an $(A,H)$-comodule algebra, that $\Omega^1(P)$ is also left $A$-covariant, and that $\Pi$ is also a left $A$-comodule map. In this case, any $\mathcal{F}$ of the form $P \, \square_H V$ has an obvious left $A$-comodule structure, and the associated connection $\nabla$ is a left $A$-comodule map.

\subsection{Drinfeld--Jimbo quantum groups}\label{section2.4}

Let $\mathfrak{g}$ be a finite-dimensional complex semisimple Lie algebra of rank $l$. We fix a Cartan subalgebra $\mathfrak{h}$ and choose a set of simple roots $\Pi = \{\alpha_1, \dots, \alpha_l\}$ for the corresponding root system. Denote by $(\cdot,\cdot)$ the symmetric bilinear form induced on $\mathfrak{h}^*$ by the Killing form of $\mathfrak{g}$, normalised so that any shortest simple root $\alpha_i$ satisfies $(\alpha_i,\alpha_i) = 2$. The Cartan matrix $(a_{ij})$ of $\mathfrak{g}$ is defined by
$a_{ij} := \big(\alpha_i^{\vee},\alpha_j\big)$,
where $\alpha_i^{\vee} := 2\alpha_i/(\alpha_i,\alpha_i)$. The set of {\em fundamental weights} $\{\varpi_1, \dots, \varpi_l\}$ of $\mathfrak{g}$ is the dual basis of simple coroots $\big\{\alpha_1^{\vee},\dots, \alpha_l^{\vee}\big\}$. We denote by $\mathcal{P}$ the $\mathbb{Z}$-span of the fundamental weights, and by
${\mathcal P}^+$ the $\mathbb{Z}_{\geq 0}$-span of the fundamental weights.

Let $q \in \mathbb{R}$ such that $q \neq -1,0,1$, and denote $q_i := q^{(\alpha_i,\alpha_i)/2}$. The {\em Drinfeld--Jimbo quantised enveloping algebra} $U_q(\mathfrak{g})$ is the noncommutative associative algebra generated by the elements $E_i$, $F_i$, $K_i$, and $K^{-1}_i$, for $ i=1, \ldots, l$, subject to the relations
\begin{gather*}
 K_iE_j = q_i^{a_{ij}} E_j K_i, \qquad K_iF_j= q_i^{-a_{ij}} F_j K_i, \qquad K_i K_j = K_j K_i, \qquad K_iK_i^{-1} = K_i^{-1}K_i = 1,\\
 E_iF_j - F_jE_i = \delta_{ij}\frac{K_i - K^{-1}_{i}}{q_i-q_i^{-1}},
\end{gather*}
and the quantum Serre relations which we omit (see \cite[Section~6.1.2]{KSLeabh} for details). The formulae
\begin{gather*}
\Delta(K_i) = K_i \otimes K_i, \qquad \Delta(E_i) = E_i \otimes K_i + 1 \otimes E_i, \qquad \Delta(F_i) = F_i \otimes 1 + K_i^{-1} \otimes F_i
\end{gather*}
define a Hopf algebra structure on $U_q(\mathfrak{g})$, satisfying $\e(E_i) = \e(F_i) = 0$, and $\e(K_i) = 1$.

For any $U_q(\mathfrak{g})$-module $V$, a vector $v\in V$ is called a \emph{weight vector} of \emph{weight}~$\mathrm{wt}(v) \in \mathcal{P}$ if
$K_i \triangleright v = q^{(\alpha_i,\mathrm{wt}(v))} v$, for all $i=1,\ldots,r$.
For each $\lambda \in\mathcal{P}^+$ there exists an irreducible finite-dimensional $U_q(\mathfrak{g})$-module $V_\lambda$, uniquely defined by the existence of a weight vector $v_{\mathrm{hw}}\in V_\lambda$ with weight $\lambda$, called a {\em highest weight vector}, satisfying
$E_i \triangleright v_{\mathrm{hw}}=0$,
for all $i = 1, \dots, l$. We call any finite direct sum of such $U_q(\mathfrak{g})$-representations a {\em type-$1$ module}.
Each type-$1$ module is spanned by weight vectors.

Let $V$ be a finite-dimensional $U_q(\mathfrak{g})$-module, $v \in V$, and $f \in V^*$, the linear dual of $V$. Consider the function
\begin{gather*}
c^{\textrm{\tiny $V$}}_{f,v}\colon \ U_q(\mathfrak{g}) \to \mathbb{C}, \qquad X \mapsto f\big(X(v)\big).
\end{gather*}
Consider now the Hopf subalgebra of $U_q(\mathfrak{g})^\circ$, the Hopf dual of $U_q(\mathfrak{g})$, generated by all functions of the form $c^{\textrm{\tiny $V$}}_{f,v}$, for $V$ a type-$1$ representation. We denote this Hopf algebra by $\mathcal{O}_q(G)$ and call it the {\em Drinfeld--Jimbo quantum coordinate algebra of~$G$}, where~$G$ is the compact, simply-connected, simple Lie group having $\mathfrak{g}$ as its complexified Lie algebra.

\subsection{Quantum flag manifolds}\label{section2.5}

For $S$ a proper subset of simple roots, consider the Hopf subalgebra
\begin{gather*}
U_q(\mathfrak{l}_S) := \big\langle K_i, E_s, F_s \,|\, i = 1, \ldots, l = \mathrm{rank}(\mathfrak{g}); \, s \in S \big\rangle \subseteq U_q(\mathfrak{g}).
\end{gather*}
The left action of $U_q(\mathfrak{g})$ on $\mathcal{O}_q(G)$ restricts to a left $U_q(\mathfrak{l}_S)$-action. We call the invariant subspace
\begin{gather*}
\mathcal{O}_q(G/L_S) := {}^{U_q(\mathfrak{l}_S)}\mathcal{O}_q(G),
\end{gather*}
the \emph{quantum flag manifold} associated to $S$. We note that $\mathcal{O}_q(G/L_S)$ is a left $\mathcal{O}_q(G)$-comodule algebra by construction. Moreover, $\mathcal{O}_q(G)$ is faithfully flat as a right $\mathcal{O}_q(G/L_S)$-module (see for example \cite[Section~5.4]{GAPP}). It now follows from \cite[Theorem~1]{Tak} that $\mathcal{O}_q(G/L_S)$ coincides with the space of right coinvariants of the coaction $\Delta_R := ( \mathrm{id}\otimes\pi_S) \circ \Delta$, where
\[
\pi_S\colon \ \mathcal{O}_q(G) \to \mathcal{O}_q(L_S) := \mathcal{O}_q(G)/\mathcal{O}_q(G/L_S)^+\mathcal{O}_q(G)
\]
is the canonical Hopf algebra projection, with $\mathcal{O}_q(G/L_S)^+:=\mathcal{O}_q(G/L_S) \cap \ker(\e)$.

Similarly we can define the Hopf subalgebra
\begin{align*}
\usslevi := \big< K_s, E_s, F_s \,|\, s \in S \big> \subseteq U_q(\mathfrak{l}_S) \subseteq U_q(\mathfrak{g}).
\end{align*}
We call the invariant subspace
\begin{align*}
\mathcal{O}_q(G/L^{\mathrm{s}}_S):= {}^{U_q(\mathfrak{l}^{\,\mathrm{s}}_S)} \mathcal{O}_q(G)
\end{align*}
the \emph{quantum Poisson homogeneous space} associated to $S$.

\subsection{Relative line modules}\label{section2.6}

From here on, we will restrict our attention to the irreducible quantum flag manifolds $\mathcal{O}_q(G/L_S)$ (see Appendix~\ref{appendixA} for the definition of irreducible). In this case $\mathcal{O}_q(L_S)$ contains a Hopf subalgebra isomorphic to $\mathcal{O}(U_1)$, the Hopf algebra of Laurent polynomials.
Moreover, the left $\mathcal{O}_q(L_S)$-coaction on $\mathcal{O}_q(G)$ restricts to an $\mathcal{O}(U_1)$-coaction on $\mathcal{O}_q(G/L^{\,\mathrm{s}}_S)$, which is to say
\[
\Delta_{R}(\mathcal{O}_q(G/L^{\mathrm{s}}_S)) \subseteq \mathcal{O}_q(G) \otimes \mathcal{O}(U_1).
\]
This gives $\mathcal{O}_q(G/L^{\mathrm{s}}_S)$ the structure of a right $\mathcal{O}(U_1)$-comodule algebra. We denote the associated $\mathbb{Z}$-algebra grading by
\begin{gather}\label{eqn:linebundeldecomp}
\mathcal{O}_q(G/L^{\mathrm{s}}_S) =: \bigoplus_{k\in \mathbb{Z}} \mathcal{E}_k.
\end{gather}
This grading is strong, as explained in \cite[Section~4]{GAPP}, meaning that each $\mathcal{E}_k$ is a line module in the sense of \cite[Section~3.5]{BeggsMajid:Leabh}. Moreover, we see that each $\mathcal{E}_k$ admits a natural presentation as an associated module to our $\mathcal{O}(U_1)$-comodule algebra (see~\cite[Section~5]{GAPP} for a detailed discussion).

By construction each $\mathcal{E}_k$ is a left $\mathcal{O}_q(G)$-subcomodule of $\mathcal{O}_q(G)$, and so, it is a relative Hopf module in the sense of Takeuchi (see Section~\ref{subsection:Takeuchi}). Hence, following \cite[Section~3.1]{DOKSS} and \cite[Section~2.3]{GAPP} we refer to each $\mathcal{E}_k$ as a \emph{relative line module}.

We can give an explicit description of this grading as follows: let $\alpha_x$ be the simple root not contained in $S$, and let $\varpi_x$ be the corresponding simple weight. Choose a weight basis $\{v_i\}_{i=1}^{N}$ of $V_{\varpi_x}$ such that $v_{N}$ is a highest weight vector. For $w_0$ the longest element of the Weyl group of~$\mathfrak{g}$, we denote the dual basis of $V_{-w_0(\varpi_x)}$ by $\{f_i\}_{i=1}^{N}$.
From \cite[Theorem~4.1]{Stok}, we know that a~set of generators for $\mathcal{O}_q(G/L^{\,\mathrm{s}}_S)$ is given by
\begin{gather*}
z_{i} := c^{V_{\varpi_x}}_{f_i,v_N}, \qquad \overline{z}_j := c^{V_{-w_0(\varpi_x)}}_{v_j,f_N} \qquad \text{for } i,j = 1, \dots, N := \dim(V_{\varpi_x}).
\end{gather*}
The $\mathbb{Z}$-grading on $\mathcal{O}_q\big(G/L^{\,\mathrm{s}}_S\big)$ is completely determined by $\deg(z_1) = 1$, and $\deg(\overline{z}_i) = -1$, for all $i = 1, \dots, N$.

\section{Bimodule connections for the relative line modules}\label{section3}

In this section we give a simple accompaniment to Beggs and Majid's criteria for a strong principal connection to induce bimodule connections. We then recall the recent classification of $\mathcal{O}_q(G)$-covariant connections $\nabla\colon \mathcal{E}_k \to \Omega^1_q(G/L_S) \otimes_{\mathcal{O}_q(G/L_S)} \mathcal{E}_k$. We also recall the recent quantum principal bundle description of the Heckenberger--Kolb calculi~\cite{CDOBBW}. Building directly on this material, we then prove the main result of the paper, namely that $\nabla$, and its holomorphic and anti-holomorphic summands, are bimodule connections. We also make some comments on the associated tensor product connections.

 \subsection{Bimodule principal connections}\label{section3.1}

 The question of when a principal connection determines a bimodule connection was thoroughly examined by Beggs and Majid in \cite[Section~5.4]{BeggsMajid:Leabh} and \cite[Proposition~5.54]{BeggsMajid:Leabh}, and sufficient conditions were given for a strong principal connection $\Pi$ to induce a bimodule connection for any associated module. We now prove a simple accompaniment to these results, suited to our needs below. It takes a weaker set of assumptions and allows for an elementary self-contained proof.

\begin{Theorem}\label{thm:theTheorem}
Let $\big(P, \Omega^1(P)\big)$ be a quantum principal bundle, endowed with a strong principal connection $\Pi\colon \Omega^1(P) \to \Omega^1(P)$ that is additionally assumed to be a right $B$-module map.
\begin{enumerate}\itemsep=0pt
 \item[$1.$] The associated connection $\nabla\colon \mathcal{F} \to \Omega^1(B) \otimes_B \mathcal{F}$, for each associated module $\mathcal{F}$, is a~bimodule connection.
 \item[$2.$] The associated bimodule map $\sigma\colon \mathcal{F} \otimes_B \Omega^1(B) \to \Omega^1(B) \otimes_B \mathcal{F}$ is an injective right $H$-comodule map. Moreover, if $\big(P, \Omega^1(P)\big)$ is strong then $\sigma$ is an isomorphism.
 \item[$3.$] Assume that $P$ is endowed with a left $A$-coaction making it a left $A$-comodule algebra, and moreover, assume that $\Omega^1(P)$ is a left $A$-covariant fodc. If $\Pi$ is a left $A$-comodule map, then $\sigma$ is a left $A$-comodule map.
 \end{enumerate}
\end{Theorem}
\begin{proof}
1.~Writing $\mathcal{F}$ as $P\square_H V$, for $V\in {}^H\mathrm{Mod},$ we let $f= \sum_i p_i\otimes v_i \in \mathcal{F}$ and $\omega \in \Omega^1(B)$, we claim that a $B$-bimodule map $\sigma\colon \mathcal{F} \otimes_B \Omega^1(B) \to \Omega^1(B) \otimes_B \mathcal{F}$ is given by
\begin{gather*}
 \sigma\left(f \otimes \omega\right)= j^{-1}\bigg(\sum_i p_{i}\omega \otimes v_{i}\bigg) .
\end{gather*}
Indeed, since $\Pi$ is a strong connection, we have that
\begin{gather*}
 \sum_i p_{i}\omega= \sum_i (\mathrm{id} - \Pi)p_{i}\omega\in \Omega^1(B) P,
\end{gather*}
and so, the map $\sigma$ is well defined.
Now we prove that $(\nabla, \sigma)$ is a bimodule connection by showing that expression \eqref{eqn:BiC} holds. We have the following:
\begin{align*}
\nabla(fb)-\nabla(f)b & = \sum_i j^{-1} \big((\mathrm{id} - \Pi)\mathrm{d}( p_{i}b) \otimes v_i \big)-\sum_i j^{-1} \big((\mathrm{id} - \Pi)\mathrm{d} p_{i} \otimes v_i\big)b
\\
&=\sum_i j^{-1} \big((\mathrm{id} - \Pi)\mathrm{d}( p_{i}b) \otimes v_i \big)-\sum_i j^{-1} \big((\mathrm{id} - \Pi)\mathrm{d}(p_{i})b \otimes v_i\big).
\end{align*}
By the Leibniz rule this reduces to
\begin{gather*}
 \sum_i j^{-1} \big((\mathrm{id} - \Pi) p_{i}\mathrm{d} b\otimes v_{i}\big) =\sum_i j^{-1} \big(p_{i}\mathrm{d} b\otimes v_{i}\big)= \sigma (f\otimes \mathrm{d} b),
\end{gather*}
where the last identity follows from the fact that each summand $p_{i}\mathrm{d} b$ is an horizontal form.

2. It is clear from the above formula that $\sigma$ is injective. Covariance of the fodc and the fact that both $j^{-1}$ and $\Pi$ are left $H$-comodule maps, imply that $\sigma$ is a right $H$-comodule map. If $\big(P, \Omega^1(P)\big)$ is strong, then we have an isomorphism
\begin{align*}
\mu'\colon \ \Omega^1(B) \otimes_B P \to P\Omega^1(B)
\end{align*}
given by multiplication in $\Omega^1(P)$. We have an additional natural embedding
\begin{gather*}
j'\colon \ \mathcal{F} \otimes_B \Omega^1(B)\hookrightarrow P\Omega^1(B) \, \square_H V,
\end{gather*}
induced by the multiplication map. The composition of isomorphisms $j'^{-1}\circ\mu$ then gives an isomorphism between $\Omega^1(B) \otimes_B \mathcal{F} $ and $ \mathcal{F} \otimes_B \Omega^1(B)$, which we see is an inverse for $\sigma$.

3. In this case, as discussed in Section~\ref{subsection:ConnFromPrinCs}, the associated connection $\nabla$ will be a left $A$-comodule map. Hence it follows directly from the definition of $\sigma$ that it will also be a left $A$-comodule map.
\end{proof}

\subsection{Relative line module connections for the Heckenberger--Kolb calculi}\label{section3.2}

As established in the seminal paper \cite{HKdR}, there exist, up to isomorphism, exactly two finite-dimensional irreducible left $\mathcal{O}_q(G)$-covariant fodc over every irreducible quantum flag manifold. We denote these calculi by
\begin{gather*}
\big(\Omega^{(1,0)}_q(G/L_S), \partial\big), \qquad \big(\Omega^{(0,1)}_q(G/L_S),\overline{\partial}\big).
\end{gather*}
Moreover, we call their direct sum the \emph{Heckenberger--Kolb calculus} on $\mathcal{O}_q(G/L_S)$ and we denote it by
\[
\big(\Omega^1_q(G/L_S),\mathrm{d}:=\overline{\partial}+\partial\big) .
\]

The task of describing the left $\mathcal{O}_q(G)$-covariant connections $\nabla\colon \mathcal{E}_k \to \Omega^1_q(G/L_S) \otimes \mathcal{E}_k$ was recently undertaken by the second author and D\'iaz Garc\'ia, Krutov, Somberg, and Strung. The proof of the following result uses the theory of principal comodule algebras~\cite{TBGS} to a construct universal connection, and then quotients to the special case of the Heckenberger--Kolb calculi.

\begin{Theorem}[{\cite[Theorem 4.5]{HVBQFM}}] \label{thm:CovConn}
Let $\mathcal{O}_q(G/L_S)$ be an irreducible quantum flag manifold, then for every $k \in \mathbb{Z}$, the relative line module $\mathcal{E}_k$ admits a unique $\mathcal{O}_q(G)$-covariant connection $\nabla\colon \mathcal{E}_k \to \Omega^1_q(G/L_S) \otimes_{\mathcal{O}_q(G/L_S)} \mathcal{E}_k$.
\end{Theorem}

In addition to the connection $\nabla$, we can also consider the \emph{holomorphic}, and respectively \emph{anti-holomorphic, connections}
\begin{gather*}
\partial_{\mathcal{E}_k}:= \big(\mathrm{proj}^{(1,0)}\otimes \mathrm{id} \big)\circ\nabla \qquad \mathrm{and} \qquad \overline{\partial}_{\mathcal{E}_k}:=\big(\mathrm{proj}^{(0,1)}\otimes \mathrm{id}\big)\circ\nabla,
\end{gather*}
where $\mathrm{proj}^{(1,0)}$ denotes the projection onto $\Omega^{(1,0)}_q(G/L_S)$, and similarly for $\mathrm{proj}^{(0,1)}$. It clearly follows from Theorem~\ref{thm:CovConn} that $\partial_{\mathcal{E}_k}$, and $\overline{\partial}_{\mathcal{E}_k}$, are the unique covariant connections for $\mathcal{E}_k$ with respect to the fodc $\Omega^{(1,0)}_q(G/L_S)$, and respectively $\Omega^{(0,1)}_q(G/L_S)$.

\subsection{A quantum principal bundle} \label{subsection:QPB}

In \cite{HKdR}, with a view to calculating higher order relations, the fodc $\Omega^{(1,0)}_q(G/L_S)$ was obtained as the restriction of a~left~$\mathcal{O}_q(G)$-covariant fodc over the quantum Poisson homogeneous space $\mathcal{O}_q(G/L^{\mathrm{s}}_S)$. The explicit construction involved a detailed $R$-matrix argument, and we refer the interested reader to \cite{CDOBBW} for a summary in the notation of this paper. The following theorem summarises the details relevant to this paper. The statement of the theorem makes use of~$S_q[G/L_S]$, the \emph{quantum homogeneous coordinate ring} of $G/L_S$, which is to say, the sub-algebra of~$\mathcal{O}_q(G)$ generated by the elements~$z_i$, for $i=1, \dots, N$.

\begin{Theorem}[{\cite[Section~3.2]{HKdR}}]
There exists a left $\mathcal{O}_q(G)$-covariant fodc
$\Omega^{(0,1)}_q\big(G/L^{\mathrm{s}}_S,\adel\big)$
such that
\begin{enumerate}\itemsep=0pt
\item[$(1)$] $\adel p = 0$, for all $p \in S_q[G/L_S]$,
\item[$(2)$] the fodc is right $\mathcal{O}(U_1)$-covariant,
\item[$(3)$] the pair $\big(\mathcal{O}_q\big(G/L^{\mathrm{s}}_S\big),   \Omega^{(0,1)}_q\big(G/L^{\mathrm{s}}_S\big)\big)$ is a quantum principal bundle,
\item[$(4)$] the zero map is a strong principal connection for the bundle.
\end{enumerate}
\end{Theorem}

The first point of the theorem follows from the construction of the fodc given in \cite[Section~3.2]{HKdR}. The right $\mathcal{O}(U_1)$-covariance was established by the authors and D\'iaz Garc\'ia in \cite[Lemma~5.5]{CDOBBW}. The quantum principal bundle property was also established by the authors and D\'iaz Garc\'ia in \cite[Proposition~5.6]{CDOBBW}, where the general theory of principal pairs (as introduced in~\cite{BwGrass,GAPP}) was used. An analogous presentation of $\Omega^{(0,1)}_q(G/L_S)$ was also given in \cite[Section~3.2]{HKdR}, and the right $\mathcal{O}(U_1)$-covariance, and the quantum principal bundle property also holds for this fodc.

\subsection{The bimodule property}\label{section3.4}

In this section we observe that the recent results collected in the subsections above are enough to imply that the line module connections are bimodule connections. We also show that the associated bimodule maps are invertible, we necessitates the following preliminary lemma.

\begin{lem}The quantum principal bundles
\begin{gather*}
\big(\mathcal{O}_q\big(G/L_S^{\mathrm{s}}\big),  \Omega^{(0,1)}_q\big(G/L_S^{\mathrm{s}}\big)\big), \qquad \big(\mathcal{O}_q\big(G/L_S^{\mathrm{s}}\big),   \Omega^{(1,0)}_q\big(G/L_S^{\mathrm{s}}\big)\big),
\end{gather*}
are strong.
\end{lem}
\begin{proof}
It was shown in \cite[Lemma 5.7]{CDOBBW} that
\[
\qphs\Omega^{(0,1)}_q(G/L_S) \subseteq \Omega^{(0,1)}_q(G/L_S)\qphs.
\]
We prove this lemma by showing that an analogous argument establishes the opposite inclusion.

An arbitrary form in $\Omega^{(0,1)}_q(G/L_S)$ is a sum of elements of the form $b'\,\adel b$, for some $b',b \in \mathcal{O}_q(G/L_S)$. Consider now the form
\begin{gather*}
b'\big(\adel b\big)z_i \in \Omega^{(0,1)}_q(G/L_S)\qphs , \qquad \textrm{for any } i = 1, \dots, N.
\end{gather*}
It follows from the observations of Section~\ref{textsection:genDets} that there exist elements $f_j'\in S_q[G/L_S]$ and $v_j' \in S_q\big[G/L^{\mathrm{op}}_S\big]$ such that $\sum_{j=1}^N v_j'f_j'=1$. Then it follows from the Leibniz rule that
\begin{gather*}
b'\big(\adel b\big)\overline{z}_i = b'\adel(b\overline{z}_i) - b'b\adel(\overline{z}_i) = \sum_{j=1}^N b'v'_jf'_j\adel(b\overline{z}_i) - \sum_{j=1}^N b'bv'_jf'_j\adel(\overline{z}_i).
\end{gather*}
Moreover, since $\adel f_j = 0$, we have that
\begin{gather*}
b'\big(\adel b\big)\overline{z}_i = \sum_{j=1}^N b'v'_j\adel(f'_jb\overline{z}_i) - \sum_{j=1}^N b'bv'_j\adel(f'_j\overline{z}_i).
\end{gather*}
Thus we see that $b'(\adel b)\overline{z}_i$ is an element of $\qphs\Omega^{(0,1)}_q(G/L_S)$, for any $i = 1, \dots, N$. Thus we can conclude that
\begin{gather*}
\qphs \Omega^{(0,1)}_q(G/L_S) \subseteq \Omega^{(0,1)}_q(G/L_S)\qphs,
\end{gather*}
and hence that the quantum principal bundle is strong.
\end{proof}

\begin{Theorem}For each line module $\mathcal{E}_k$, the connections $\nabla, \partial_{\mathcal{E}_k}$ and $\overline{\partial}_{\mathcal{E}_k}$ are bimodule connections. Moreover, the associated bimodule maps are left $\mathcal{O}_q(G)$-comodule isomorphisms.
\end{Theorem}
\begin{proof}Recall from Section~\ref{subsection:QPB} that the zero map on $\Omega^1_q\big(G/L^{\,\mathrm{s}}_S\big)$ is a strong principal connection that realises the line modules connections $\adel_{\mathcal{E}_k}$ as associated connections. Since the zero map is clearly a right $\mathcal{O}_q(G/L_S)$-module map, we can conclude from Theorem \ref{thm:theTheorem} that each $\adel_{\mathcal{E}_k}$ is a~bimodule connection. Moreover, since the zero map is clearly a right $\mathcal{O}_q(G)$-comodule map, and by the above lemma the quantum principal bundle is strong, we can conclude from Theorem~\ref{thm:theTheorem} that the associated bimodule map $\sigma_k$ is a left $\mathcal{O}_q(G)$-comodule isomorphism.

A similar argument establishes that the connection $\del_{\mathcal{E}_k}$ is a bimodule conection with an invertible bimodule map. Moreover, these two results together imply that $\nabla$ is a bimodule connection with an invertible bimodule map.
\end{proof}

\begin{Remark}The decomposition of the Heckenberger--Kolb calculus $\Omega^1_q(G/L_S)$ into its $(1,0)$ and $(0,1)$ summands extends to a complex structure $\Omega^{(\bullet,\bullet)}$ on the maximal prolongation of $\Omega^1_q(G/L_S)$. Moreover, as shown in \cite[Theorem~4.5]{HVBQFM}, with respect to the complex structure $\Omega^{(\bullet,\bullet)}$ the $(0,1)$-connections $\adel_{\mathcal{E}_k}$ extend to holomorphic structures for the line bundles $\mathcal{E}_k$. Thus we see that each $\big(\mathcal{E}_k,\adel_{\mathcal{E}_k}\big)$ is a \emph{holomorphic bimodule} in the sense of Beggs and Majid \cite[Definition~7.14]{BeggsMajid:Leabh}.
\end{Remark}

\subsection{Tensor products of connections} \label{subsection:tensorprods}

For $k \in \mathbb{Z}_{>0}$, let $\mathcal{E}_k$ be the corresponding line module over the irreducible quantum flag manifold $\mathcal{O}_q(G/L_S)$. Since the \mbox{$\mathbb{Z}$-grading} given in \eqref{eqn:linebundeldecomp} is strong, it follows from \cite[Corollary 3.1]{FreddyO}, that multiplication gives an isomorphism
\begin{align*}
\mathcal{E}_k \simeq \underbrace{\mathcal{E}_1 \otimes_{\mathcal{O}_q(G/L_S)} \cdots \otimes_{\mathcal{O}_q(G/L_S)} \mathcal{E}_1}_{k\text{-times}},
\end{align*}
which is of course a left $\mathcal{O}_q(G)$-comodule map. Since the connection $\nabla$ on $\mathcal{E}_1$ is a bimodule connection, we can build a bimodule connection for $\mathcal{E}_k$ through repeated iterations of formula~\eqref{eqn:bimcontenpr}.

\begin{prop}The connections
\begin{enumerate}\itemsep=0pt
\item[$(1)$] $\adel_{\mathcal{E}_{k + l}}\colon \mathcal{E}_{k + l} \to \Omega^{(0,1)}_q(G/L_S) \otimes_{\mathcal{O}_q(G/L_S)} \mathcal{E}_{k +l}$,
\item[$(2)$] $\adel_{\mathcal{E}_{k} \otimes \mathcal{E}_l}\colon \mathcal{E}_k \otimes \mathcal{E}_l \to \Omega^{(0,1)}_q(G/L_S) \otimes_{\mathcal{O}_q(G/L_S)} \mathcal{E}_k \otimes_{\mathcal{O}_q(G/L_S)} \mathcal{E}_l$
\end{enumerate}
are isomorphic. Moreover, the corresponding holomorphic connections are isomorphic.
\end{prop}
\begin{proof}Since the associated bimodule map $\sigma_1$ is a left $\mathcal{O}_q(G)$-comodule map, the tensor product connection for $\mathcal{E}_k \otimes_{\mathcal{O}_q(G/L_S)} \mathcal{E}_l$ is left $\mathcal{O}_q(G)$-covariant. Theorem \ref{thm:CovConn} implies that this connection must be isomorphic to $\adel_{\mathcal{E}_{k+l}}$. An analogous construction exists for the line modules $\mathcal{E}_{-k}$, for all $k \in \mathbb{Z}$, and for the holomorphic connections $\del_{\mathcal{E}_{k+l}}$ and $\del_{\mathcal{E}_{k} \otimes \mathcal{E}_l}$.
\end{proof}


\section{Some explicit descriptions of the associated bimodule maps}\label{section4}

In this section we give three complementary presentations of the bimodule map associated to the connection $\adel_{\mathcal{E}_k}$, first in terms of generalised quantum determinants, then in terms of the FRT presentation of $\mathcal{O}_q(G)$, and finally in terms of Takeuchi's categorical equivalence for relative Hopf modules.

\subsection{Generalised determinant relations} \label{textsection:genDets}

The elements $z_{N}^k$, and $\overline{z}_{N}^k$, are highest weight vectors of weight $k\varpi_s$, and $-kw_0(\varpi_s)$. Thus we have a dual pair of irreducible $U_q(\mathfrak{g})$-modules,
\begin{gather*}
V_k:= U_q(\mathfrak{g})z^k_N, \qquad {}^{\vee}V_k := U_q(\mathfrak{g})\overline{z}_N^k.
\end{gather*}
Since both $V_k \otimes V_k^{\vee}$ and ${}^{\vee}V_k \otimes V_k$ contain a copy of the trivial module, we must have a family of elements
\begin{gather*} 
\{f_{i}, \, f'_{i}\}_{i=1}^{N} \subseteq V_k, \qquad \{v_i, \, v'_i\}_{i=1}^{N} \subseteq {}^{\vee}V_k,
\end{gather*}
satisfying
\begin{gather}\label{eqn:genDet}
\sum_{i=1}^{N} f_iv_i = 1, \qquad \sum_{i=1}^{N} v'_if'_i = 1.
\end{gather}
It follows from the construction of the Heckenberger--Kolb calculus $\Omega^{(0,1)}_q(G/L_S^{\,\mathrm{s}})$ that
\begin{align} \label{eqn:holof}
\adel f_i = \adel f'_i = 0, \qquad \textrm{for all } i = 1, \dots, N,
\end{align}
a fact that will be used in the proof of the proposition below. (Indeed this fact is a crucial ingredient in the proof of the noncommutative Borel--Weil theorem for $\mathcal{O}_q(G/L_S)$ given in \cite{CDOBBW}.)

\begin{prop}
For $k \in \mathbb{Z}_{> 0}$, it holds that
\begin{enumerate}\itemsep=0pt
\item[$1.$] The bimodule map
\[
\sigma_k\colon \ \mathcal{E}_k \otimes_{\mathcal{O}_q(G/L_S)} \Omega^{(0,1)}_q(G/L_S) \to \Omega^{(0,1)}_q(G/L_S) \otimes_{\mathcal{O}_q(G/L_S)} \mathcal{E}_{k},
\]
associated to the connection $\adel_{\mathcal{E}_{k}}$, satisfies
\begin{gather*}
\sigma_k\big(e \otimes \adel b\big) = \sum_{i=1}^{N} \adel(ebv'_i) \otimes f'_i - \sum_{i=1}^{N} \adel(ev'_i) \otimes f'_ib.
\end{gather*}
Moreover, the inverse map $\sigma^{-1}$ satisfies
\begin{gather*}
\sigma^{-1}_k\big(\adel b \otimes e\big) = \sum_{i=1}^N f_i \otimes \adel(v_ibe) - \sum_{i=1}^N bf_i \otimes \adel(v_ie),
\end{gather*}

\item[$2.$] The bimodule map
\[
\sigma_{-k}\colon \ \mathcal{E}_{-k} \otimes_{\mathcal{O}_q(G/L_S)} \Omega^{(0,1)}_q(G/L_S) \to \Omega^{(0,1)}_q(G/L_S) \otimes_{\mathcal{O}_q(G/L_S)} \mathcal{E}_{-k},
\]
associated to $\adel_{\mathcal{E}_{-k}}$, satisfies
\begin{gather*}
\sigma_{-k}\big(e \otimes \adel b\big) = \sum_{i=1}^{N} \adel(ebf_i) \otimes v_i - \sum_{i=1}^{N} \adel(ef_i) \otimes v_ib.
\end{gather*}
Moreover, the inverse map $\sigma^{-1}_{-k}$ satisfies
\begin{gather*}
\sigma^{-1}_{-k}\big(\adel b \otimes e\big) = \sum_{i=1}^N v'_i \otimes \adel(f'_ibe) - \sum_{i=1}^N bv'_i \otimes \adel(f'_ie).
\end{gather*}
\end{enumerate}
\end{prop}
\begin{proof}
1. For any $e \in \mathcal{E}_k$, it follows from \eqref{eqn:genDet} and \eqref{eqn:holof} that
\begin{gather*}
e\adel b =
\sum_{i=1}^{N} \adel(ebv'_i)f'_i - \sum_{i=1}^{N} \adel(ev'_i)f'_ib.
\end{gather*}
Thus we have expressed $e\adel b$ as an element of $\Omega^1_q(G/L_S)\mathcal{O}_q\big(G/L^{\,\mathrm{s}}_S\big)$, which gives us the claimed formula. The formula for the inverse map is established analogously.

2. For any $e \in \mathcal{E}_{-k}$, it follows from \eqref{eqn:genDet} and \eqref{eqn:holof} that
\begin{gather*}
e\adel b
= \sum_{i=1}^{N} \adel(ebf_i)v_i - \sum_{i=1}^{N} \adel(ef_i)v_ib.
\end{gather*}
Thus we have expressed $e\adel b$ as an element of $\Omega^1_q(G/L_S)\mathcal{O}_q(G/L^{\,\mathrm{s}})$, which gives us the claimed formula. The formula for the inverse map is established analogously.
\end{proof}

In what follows we will denote by $\sigma_k$ the bimodule map associated to the connection $\overline{\partial}_{\mathcal{E}_k}$ and, similarly, by $\sigma_{-k}$ the bimodule map associated to the connection $\adel_{\mathcal{E}_{-k}}$.

\begin{eg}
Let us now check our formula for the bimodule map $\adel_{\mathcal{E}_k}$, for $k \in \mathbb{Z}_{>0}$, by looking at the commutative case. We have
\begin{align*}
\sigma_k\big(e \otimes \adel b\big) &= \sum_{i=1}^{N} \adel(ev'_ib) \otimes f'_i - \sum_{i=1}^{N} \adel(ev'_i) \otimes f'_ib \\
&= \sum_{i=1}^{N} \adel(ev'_i)b \otimes f'_i + \sum_{i=1}^{N} ev'_i\adel b \otimes f'_i - \sum_{i=1}^{N} \adel(ev'_i) \otimes f'_ib \\
&= \sum_{i=1}^{N} \adel b \otimes ev'_if'_i \\
&= \adel b \otimes e.
\end{align*}
Thus we see that $\sigma_k$ reduces to the usual flip map, as it should. A similar situation holds for the line bundle~$\mathcal{E}_{-k}$.
\end{eg}

We finish with the interesting parallel observation that the generalised determinant identities can be used to give an explicit description of the connections $\del_{\mathcal{E}_k}$ and $\adel_{\mathcal{E}_k}$. This generalises the special case of the line module~$\mathcal{E}_1$ over quantum projective space $\mathcal{O}_q\big(\mathbb{CP}^{n}\big)$ given in \cite[Example~8.5]{MMF1}.

\begin{prop}Let $\mathcal{O}_q(G/L_S)$ be an irreducible quantum flag manifold, and $k \in \mathbb{Z}_{>0}$, then for any $e \in \mathcal{E}_k$, we have
\begin{gather*}
\del_{\mathcal{E}_k}(e) = \sum_{i=1}^N \del(ev'_i) \otimes f_i', \qquad \adel_{\mathcal{E}_{k}}(e) = \sum_{i=1}^N \adel(ev'_i) \otimes f'_i.
\end{gather*}
 For any $e \in \mathcal{E}_{-k}$, we have
\begin{gather}\label{eqn:exadelbar}
\del_{\mathcal{E}_{-k}}(e) = \sum_{i=1}^N \del(ef_i) \otimes v_i, \qquad \adel_{\mathcal{E}_{-k}}(e) = \sum_{i=1}^N \adel(ef_i) \otimes v_i.
\end{gather}
\end{prop}
\begin{proof}We give the proof for the first identity, the proofs of the other three being completely analogous. This follows from the calculation
\begin{gather*}
\del_{\mathcal{E}_k}(e) = j^{-1}(\del(e)) = j^{-1}\left(\sum_{i=1}^N \del(e)v'_if'_i\right) = j^{-1}\left(\sum_{i=1}^N \del(ev'_i)f'_i\right) = \sum_{i=1}^N \del(ev'_i) \otimes f'_i,
\end{gather*}
where as usual, we have used the fact that $\del v_i = 0$, for all $i=1, \dots, N$.
\end{proof}

\subsection{Some examples in terms of the FRT generators} \label{subsection:FRT}

The non-exceptional Drinfeld--Jimbo quantum coordinate algebras $\mathcal{O}_q(G)$ admit a well-known alternative construction, called the FRT construction \cite{FRT89}, which essentially gives a generator and relation presentation in terms of the coordinate functions of the first fundamental representation~$V_{\varpi_1}$ of~$U_q(\mathfrak{g})$. It proves useful when explicit commutation relations are required, as in Section~\ref{subsection:Takeuchi} below, and is commonly used in the $C^*$-algebra literature.

In this subsection we give an FRT presentation of the bimodule maps for the special cases of the standard Podle\'s sphere, the more general quantum Grassmannians, and the quantum quadrics, special cases of the irreducible quantum flag manifolds, as presented in Appendix~\ref{appendixA} below.

\subsubsection{The standard Podle\'s sphere}

The standard Podle\'s sphere $\mathcal{O}_q\big(S^2\big)$ is the single quantum flag manifold of $\mathcal{O}_q(\mathrm{SU}_2)$, moreover the quantum Poisson homogeneous space of $\mathcal{O}_q\big(S^2\big)$ coincides with $\mathcal{O}_q(\mathrm{SU}_2)$. Denoting by $u^i_j$, for $i,j = 1,2$, the standard generators of $\mathcal{O}_q(\mathrm{SU}_2)$ (see \cite[Section~4.1]{KSLeabh}), the quantum determinant relations
\begin{gather*}
u_{11}u_{22} - qu_{21}u_{12} = 1, \qquad u_{22}u_{11} - q^{-1}u_{12}u_{21} = 1
\end{gather*}
give a pair of relations of the form \eqref{eqn:genDet}. Thus we see that for any $e \in \mathcal{E}_1$, the element $\sigma_1\big(e \otimes \adel b\big)$ is equal to
\begin{gather*}
\adel(ebu_{22}) \otimes u_{11}- \adel(eu_{22}) \otimes u_{11} b - q^{-1} \big(\adel(ebu_{12}) \otimes u_{21}-\adel(eu_{12}) \otimes u_{21}b\big).
\end{gather*}
Moreover, for any $e \in \mathcal{E}_{-1}$, the element $\sigma_{-1}(e \otimes \adel b)$ is equal to
\begin{gather*}
\adel(ebu_{11}) \otimes u_{22}-\adel(eu_{11}) \otimes u_{22}b - q\big( \adel(ebu_{21}) \otimes u_{12}- \adel(eu_{21}) \otimes u_{12}b\big).
\end{gather*}
The line bundles $\mathcal{E}_{-2}$ and $\mathcal{E}_{2}$ are isomorphic as relative line modules to the fodc $\Omega^{(1,0)}_q\big(S^2\big)$ and $\Omega^{(0,1)}_q\big(S^2\big)$ respectively. In this case the square of the first determinant relation
\begin{gather*}
(u_{22})^2(u_{11})^2-\big(q+q^{-1}\big)u_{22}u_{12}u_{21}u_{11}+q^{-2}(u_{12})^2(u_{21})^2=1,
\end{gather*}
gives a relation of the form \eqref{eqn:genDet}. This means that for $e \in \mathcal{E}_2$, we have
\begin{align*}
\sigma_2\big(e \otimes \adel b\big) ={} & \adel\big(eb(u_{22})^2\big) \otimes (u_{11})^2 -\adel\big(e(u_{22})^2\big) \otimes (u_{11})^2b
\\&{}- \big(q+q^{-1}\big)\big(\adel(ebu_{22}u_{12}\big) \otimes u_{21}u_{11} -\adel(ebu_{22}u_{12}) \otimes u_{21}u_{11}b)
\\&{} - q^{-2} \big(\adel\big(eb(u_{12})^2\big)\otimes (u_{21})^2-\adel\big(e(u_{12})^2\big) \otimes (u_{21})^2b\big).
\end{align*}
Moreover, we have an analogous expression for the line bundle $\mathcal{E}_{-2}$.
This reproduces the explicit presentation of the bimodule map given by Beggs and Majid in \cite[example 5.51]{BeggsMajid:Leabh}.

\subsubsection{The quantum grassmannians}

Let us now consider the $A_{n-1}$-series irreducible quantum flag manifolds, that is, the quantum Grassmannians, where the crossed simple root $\alpha_m$ corresponds to the quantum Grassmannian $\mathcal{O}_q(\mathrm{Gr}_{n,m})$. Quantum projective space $\mathcal{O}_q\big(\mathbb{CP}^{n-1}\big)$ corresponds to the special cases of $\alpha_m = \alpha_1$, and $\alpha_m = \alpha_{n-1}$. Moreover, these reduce to the standard Podle\'s sphere for the case of $\mathcal{O}_q(\mathrm{SU}_2)$.

Denote by $u_{ij}$, for $i,j = 1, \dots, n$, the FRT generators of $\mathcal{O}_q(\mathrm{SU}_n)$ \cite{FRT89}, (see also \cite[Section~9.2]{KSLeabh}). Let $I: = \{i_1, \dots, i_m\}$ and $J:=\{j_1, \dots, j_p\}$ be a pair of subsets of $\{1,\dots, n\}$. The associated {\em quantum minor} $[I|J]$ is the element of $\mathcal{O}_q(\mathrm{SU}_n)$ given by
\begin{gather*}
[I|J] := \sum_{\sigma \in S_p} (-q)^{\ell(\sigma)}u_{\sigma(i_1)j_1} \cdots u_{\sigma(i_p)j_p} = \sum_{\sigma \in S_p} (-q)^{\ell(\sigma)} u_{i_1\sigma(j_1)} \cdots u_{i_p\sigma(j_p)}.
\end{gather*}
The quantum Poisson homogeneous space $\mathcal{O}_q(S^{n,m})$ of the quantum Grassmannian is generated by the quantum minors
\begin{gather*}
z_{A} := \big[A|M^{\perp}\big], \qquad \overline{z}_B := [B|M], \qquad \textrm{for } |A| = n-m, \ |B| = m,
\end{gather*}
where $M := \{1, \dots, m\}$ and $M^\perp:= \{m+1, \dots, n\}$.

For $I$, $J$ ordered subsets of $\{1,\dots ,n\}$, we denote by $I^{\perp}$, and $J^{\perp}$, the ordered complements to $I$, and $J$ respectively, in $\{1,\dots, n\}$. Moreover, we write
\begin{gather*}
\ell(I,J):= |\{(i,j)\in I\times J \,| \, i>j\}|, \qquad \ell_{\perp}(A,B) := \ell\big(A,A^{\perp}\big) - \ell\big(B,B^{\perp}\big).
\end{gather*}
This gives the well-known antipode expression
\begin{gather*}
S([I|J]) = (-q)^{\ell_{\perp}(I,J)}\big[{J^\perp}|{I^\perp}\big],
\end{gather*}
From this we can produce the standard identity
\begin{align*}
1 = \e(z_M) &= m \circ (S \otimes \mathrm{id}) \circ \Delta(z_M)\\
& = \sum_{|A| = m} S([M,A])[A,M]\\
& = \sum_{|A| = m} q^{\ell_{\perp}(M,A)} \big[A^{\perp},M^{\perp}\big][A,M]\\
& = \sum_{|A| = m} q^{\ell_{\perp}(M,A)} \overline{z}_{A^{\perp}}z_A,
\end{align*}
which is of the form \eqref{eqn:genDet}. Now, for $e \in \mathcal{E}_1$, we see that
\begin{gather*}
\sigma_1\big(e \otimes \adel b\big) = \sum_{|A|=n-r} q^{\ell_{\perp}(M,A)} \adel(eb\overline{z}_{A^{\perp}}) \otimes z_A- \sum_{|A|=n-r} q^{\ell_{\perp}(M,A)} \adel(e\overline{z}_{A^{\perp}}) \otimes z_Ab.
\end{gather*}
A similar argument shows that, for $e \in \mathcal{E}_{-1}$,
\begin{gather*}
\sigma_{-1}\big(e \otimes \adel b\big) = \sum_{|B|=m} q^{\ell_{\perp}(M,B)} \adel(ebz_{B^{\perp}}) \otimes \overline{z}_B - \sum_{|B|=m} q^{\ell_{\perp}(M,B)} \adel(ez_{B^{\perp}}) \otimes \overline{z}_Bb.
\end{gather*}
Following the same approach we can also produce explicit formula for the connections $\adel_{\mathcal{E}_k}$. For for $e \in \mathcal{E}_1$, it holds that
\begin{gather*}
\adel_{\mathcal{E}_{-1}}(e) = \sum_{|A|=n-r} q^{\ell_{\perp}(M,A)} \adel(e\overline{z}_{A^{\perp}}) \otimes z_A.
\end{gather*}
For $e \in \mathcal{E}_{-1}$, it holds that
\begin{gather*}
\adel_{\mathcal{E}_{-1}}(e) = \sum_{|B|=m} q^{\ell_{\perp}(M,B)} \adel(ez_{B^{\perp}}) \otimes \overline{z}_B.
\end{gather*}
Analogous, if more technical, quantum minor type formulae for $\sigma_k$ and $\del$ can be produced for the general line bundles $\mathcal{E}_k$, for $k \in \mathbb{Z}$. For the special case of the standard Podle\'s sphere these formulae reduce to the presentations of the anti-holomorphic connection given by Beggs and Majid in \cite[Example~5.23]{BeggsMajid:Leabh}.

\subsubsection{The quantum quadrics}

Let us now consider the $B_n$-series irreducible quantum flag manifolds, which is to say, the \emph{odd quantum quadrics} $\mathcal{O}_q(\mathbf{Q}_{2n+1})$. In this case the Poisson quantum homogeneous space $q$-deforms the coordinate functions of the Stieffel manifold $V_2\mathbb{R}^{2n+1} \simeq \mathrm{SO}_{2n+1}/\mathrm{SO}_{2n-1}$. Thus it is called \emph{quantum real $(2n+1,2)$-Stieffel manifold} and denoted by $\mathcal{O}_q\big(V_2\mathbb{R}^{2n+1}\big)$.

For the $B$-series, the FRT algebra is a $q$-deformation of the coordinate algebra of $\mathrm{SO}_{2n+1}$. In other words, it is the proper subalgebra of $\mathcal{O}_q(\mathrm{Spin}_{2n+1})$ generated by the coordinate functions of the first fundamental representation $V_{\varpi_1}$ of $U_q(\mathfrak{so}_{2n+1})$. We write $u_{ij}$, where $i,j = 1, \dots, 2n+1$, for the FRT generators of $\mathcal{O}_q(\mathrm{SO}_{2n+1})$. Denoting
\begin{gather*}
i' = 2n+2-i, \qquad \rho_i := n + \frac{1}{2} -i, \quad \textrm{for } i < i', \qquad \rho_{i} := -\rho_{i'}, \quad \textrm{for } i' < i,
\end{gather*}
an explicit formula for the antipode of $\mathcal{O}_q(\mathrm{SO}_{2n+1})$ is given by
\begin{gather*}
S(u_{ij}) = q^{\rho_i - \rho_j} u_{j'i'},
\end{gather*}

The quantum Stieffel manifold $\mathcal{O}_q\big(V_2\mathbb{R}^{2n+1}\big)$ is a subalgebra of $\mathcal{O}_q(\mathrm{SO}_{2n+1})$, with generators $z_i = u_{i,2n+1}$, and $\overline{z}_i = u_{i1}$, for $i = 1, \dots, 2n+1$. The identity
\begin{gather*}
1 = \e(z_{2n+1}) = m \circ (S \otimes \mathrm{id}) \circ \Delta(z_{2n+1}) = \sum_{a=1}^{2n+1} q^{\rho_i - n - \frac{1}{2}}\overline{z}_{a'}z_a,
\end{gather*}
follows from the antipode formula, and implies that, for $e \in \mathcal{E}_1$,
\begin{gather*}
\sigma_1\big(e \otimes \adel b\big) = \sum_{i=1}^{2n+1} q^{-(\rho_i + n + \frac{1}{2})} \adel(eb\overline{z}_{i'}) \otimes z_i - \sum_{i=1}^{2n+1} q^{-(\rho_i + n + \frac{1}{2})} \adel(e\overline{z}_{i}) \otimes z_ib.
\end{gather*}
An analogous argument shows that, for $e \in \mathcal{E}_{-1}$, we have
\begin{gather*}
\sigma_{-1}\big(e \otimes \adel b\big) = \sum_{i=1}^{2n+1} q^{\rho_i - n - \frac{1}{2}}\adel(ebz_{i'}) \otimes \overline{z}_i - \sum_{i=1}^{2n+1} q^{\rho_i - n - \frac{1}{2}} \adel(ez_{i}) \otimes \overline{z}_ib.
\end{gather*}
Following the same approach we can also produce explicit formula for the connections $\adel_{\mathcal{E}_k}$. For for $e \in \mathcal{E}_1$, it holds that
\begin{gather*}
\adel_{\mathcal{E}_1}(e) = \sum_{i=1}^{2n+1} q^{-(\rho_i + n + \frac{1}{2})} \adel(e\overline{z}_{i'}) \otimes z_i
\end{gather*}
For $e \in \mathcal{E}_{-1}$, it holds that
\begin{gather*}
\adel_{\mathcal{E}_{-1}}(e) = \sum_{i=1}^{2n+1} q^{\rho_i - n - \frac{1}{2}}\adel(ez_{i'}) \otimes \overline{z}_i.
\end{gather*}
The higher order bundles $\mathcal{E}_k$ admit analogous descriptions, as do the \emph{even quantum quadrics} $\mathcal{O}_q(\mathbf{Q}_{2n})$, which is to say the $C_{n}$-irreducible quantum flag manifolds.

\subsection{A description using Takeuchi's equivalence} \label{subsection:Takeuchi}

Consider, for any quantum flag manifold $\mathcal{O}_q(G/L_S)$, the category of \emph{relative Hopf modules}
\[
{}^{~~~\,\mathcal{O}_q(G)}_{\mathcal{O}_q(G/L_S)}\mathrm{Mod}_{\mathcal{O}_q(G/L_S)},
\]
whose objects are left \mbox{$\mathcal{O}_q(G)$-comodules} $\Delta_L\colon \mathcal{F} \to \mathcal{O}_q(G) \otimes \mathcal{F}$, endowed with the structure of an $\mathcal{O}_q(G/L_S)$-bimodule, such that
\begin{gather*}
\Delta_L(bfb') = \Delta_L(b)\Delta_L(f)\Delta_L(b'), \qquad \textrm{for all } f \in \mathcal{F},\ b,b' \in \mathcal{O}_q(G/L_S),
\end{gather*}
and whose morphisms are both left $\mathcal{O}_q(G)$-comodule, and $\mathcal{O}_q(G/L_S)$-bimodule, maps. Moreover, we consider the full subcategory
\[
{}^{~~~\,\mathcal{O}_q(G)}_{\mathcal{O}_q(G/L_S)}\mathrm{Mod}_{0},
\]
whose objects $\mathcal{F}$ are those satisfying $\mathcal{F} \mathcal{O}_q(G/L_S)^+ = \mathcal{O}_q(G/L_S)^+\mathcal{F}$. This is a monoidal category with respect to the bimodule tensor product $\otimes_B$. Moreover, the calculi $\Omega^{(0,1)}_q(G/L_S)$ and $\Omega^{(1,0)}_q(G/L_S)$ live in this subcategory \cite[Proposition 3.3]{HKdR}, as do the relative line modules $\mathcal{E}_k$, as explained for example in \cite[Appendix A.3]{GAPP}.

Consider next the category
\[
{}^{\mathcal{O}_q(L_S)}\mathrm{Mod}
\]
whose objects $\Delta_L\colon V \to \mathcal{O}_q(L_S) \otimes V$ are left \mbox{$\mathcal{O}_q(L_S)$-comodules}, and whose morphisms are left $\mathcal{O}_q(L_S)$-comodule maps. This has a monoidal structure given by the usual tensor product of comodules.

We define a functor
\[
\Phi\colon \ {}^{~~~\,\mathcal{O}_q(G)}_{\mathcal{O}_q(G/L_S)}\mathrm{Mod}_{\mathcal{O}_q(G/L_S)} \to {}^{\mathcal{O}_q(L_S)}\mathrm{Mod}_{\mathcal{O}_q(G/L_S)},
\]
by setting $\Phi(\mathcal{F}) := \mathcal{F}/\mathcal{O}_q(G/L_S)^+\mathcal{F}$,
with the left $\mathcal{O}_q(L_S)$-comodule structure of $\Phi(\mathcal{\mathcal{F}})$ given by
\begin{gather*}
\Delta_L[f] := \pi_S(f_{(-1)})\otimes [f_{(0)}],
\end{gather*}
where the square brackets denote the coset of an element in $\Phi(\mathcal{\mathcal{F}})$, and morphisms descend to the quotient. A functor in the other direction
\[
\Psi\colon \ {}^{\mathcal{O}_q(L_S)}\mathrm{Mod}_{\mathcal{O}_q(G/L_S)} \to {}^{~~~\,\mathcal{O}_q(G)}_{\mathcal{O}_q(G/L_S)}\mathrm{Mod}_{\mathcal{O}_q(G/L_S)},
\]
is given by the cotensor product
$\Psi(V) := \mathcal{O}_q(G) \square_{\mathcal{O}_q(L_S)} V$
which acts on a morphism $\gamma$ as $\Phi(\gamma) = \mathrm{id} \otimes \gamma$.

Since $\mathcal{O}_q(G)$ is faithfully flat as a right $\mathcal{O}_q(G/L_S)$-module, it follows from Takeuchi's equivalence for quantum homogeneous spaces~\cite{Tak}, that $\Phi$ induces an equivalence of categories. Moreover, it is a monoidal equivalence \cite[Section~4]{MMF2}. We note that the relative line modules~$\mathcal{E}_k$ are invertible objects in the category ${}^{~~~\,\mathcal{O}_q(G)}_{\mathcal{O}_q(G/L_S)}\mathrm{Mod}_0$. Indeed, as shown in \cite[Proposition 5.7]{GAPP} for details), they are the only such objects.

Since any $\sigma_k$ is clearly a morphism, we can completely describe it by identifying its image under $\Phi$. The following proposition tells us that $\Phi(\sigma_k)$ takes quite a simple form.

For sake of clarity in the proof, we find it useful to establish a simple general lemma.

\begin{lem}
Let $\nabla\colon \mathcal{F} \to \Omega^1(B) \otimes_B\mathcal{F} $ and $\nabla'\colon \mathcal{F}' \to\Omega^1(B) \otimes_B \mathcal{F}' $ be two bimodule connections, and $\alpha\colon \mathcal{F} \to \mathcal{F}'$ an isomorphism between them. We note that the following diagram commutes
\begin{gather} \label{eqn:commdiag}
\begin{split}
& \xymatrix{
\mathcal{F} \otimes_B \Omega^1(B) \ar[d]_{\alpha \otimes \mathrm{id}} \ar[rrr]^{\sigma} & & & \Omega^1(B) \otimes_B \mathcal{F} \ar[d]^{\mathrm{id} \otimes \alpha} \\
\mathcal{F}' \otimes_B \Omega^1(B) \,\, \ar[rrr]_{\sigma'} & & & \Omega^1(B) \otimes_B \mathcal{F}',
}\end{split}
\end{gather}
where $\sigma$ and $\sigma'$ are the bimodule maps associated to $\nabla$ and $\nabla'$ respectively.
\end{lem}
\begin{proof}
The calculation
\begin{gather*}
 \sigma'(\alpha (f) \otimes \mathrm{d}b) = \nabla ' (\alpha (fb))-\nabla'(\alpha(f))b
 = (\mathrm{id} \otimes \alpha)(\nabla( fb)-\nabla(f)b)
 =(\mathrm{id} \otimes \alpha) \sigma(f \otimes \mathrm{d} b)
\end{gather*}
directly implies the result.
\end{proof}

\begin{prop} \label{prop:TAKtheta}
For any irreducible quantum flag manifold $\mathcal{O}_q(G/L_S)$, the map
\[
\Phi(\sigma_k)\colon \ \Phi(\mathcal{E}_k) \otimes \Phi\big(\Omega^{(0,1)}_q(G/L_S)\big) \to \Phi\big(\Omega^{(0,1)}_q(G/L_S)\big) \otimes \Phi(\mathcal{E}_k),
\]
for any line module $\mathcal{E}_k$, with $k \in \mathbb{Z}$, satisfies
\[
\Phi(\sigma_k)([e] \otimes [\omega]) = \theta^k [\omega] \otimes [e],
\]
for a nonzero scalar $\theta \in \mathbb{C}$.
\end{prop}

\begin{proof}
Since the functor $\Phi$ induces a monoidal equivalence, we have an isomorphism
\begin{gather*}
\Phi\big(\mathcal{E}_k \otimes_{\mathcal{O}_q(G/L_S)} \Omega^{(0,1)}_q(G/L_S)\big) \simeq \Phi(\mathcal{E}_k) \otimes \Phi\big(\Omega^{(0,1)}_q(G/L_S)\big).
\end{gather*}
As in the classical setting, for each irreducible quantum flag manifold, $\Phi\big(\Omega^{(0,1)}_q(G/L_S)\big)$ is an irreducible $U_q(\mathfrak{l}_S)$-module. Moreover, since $\mathcal{E}_k$ is invertible as a $B$-bimodule, the $U_q(\mathfrak{l}_S)$-module $\Phi(\mathcal{E}_k)$ must be one-dimensional. Thus we see that $\Phi(\mathcal{E}_k) \otimes \Phi\big(\Omega^{(0,1)}_q(G/L_S)\big)$ is an irreducible $U_q(\mathfrak{l}_S)$-module. An analogous argument establishes that the tensor product $\Phi\big(\Omega^{(0,1)}_q(G/L_S)\big) \otimes \Phi(\mathcal{E}_k)$ is irreducible as a $U_q(\mathfrak{l}_S)$-module.
If $\big[\,\adel b_{hw}\big]$ is a highest weight vector of $\Phi\big(\Omega^{(0,1)}_q(G/L_S)\big)$, then there exists a non-zero scalar~$\theta_k$, such that
\[
\Phi(\sigma_k)\big(\big[\,\adel b_{hw}\big] \otimes [e]\big) = \theta_k [e] \otimes \big[\,\adel b_{hw}\big].
\]
For any other element $\big[\,\adel b\big]$, there exists an $X\in U_q\big(\mathfrak{l}_S^{\,\mathrm{s}}\big)$ such that
\begin{gather*}
 \big[\,\adel b\big] \otimes [e]=X\big(\big[\,\adel b_{hw}\big] \otimes [e]\big).
\end{gather*}
Hence we have that
\begin{align*}
 \Phi(\sigma_k)\big(\big[\,\adel b\big] \otimes [e]\big) &= \theta_k X\big([e] \otimes \big[\,\adel b_{hw}\big]\big)
= \theta_k (X_{(1)}[e]) \otimes \big(X_{(2)}\big[\,\adel b_{hw}\big]\big)\\
&= \theta_k [e] \otimes \big(X\big[\,\adel b_{hw}\big]\big)
= \theta_k [e] \otimes \big[\,\adel b\big].
\end{align*}
Note that the penultimate equality follows from the fact that, since $\mathfrak{l}_S^{\,\mathrm{s}}$ is a semisimple Lie algebra, it admits no non-trivial one-dimensional representations, and so, the Hopf algebra $U_q\big(\mathfrak{l}_S^{\,\mathrm{s}}\big)$ admits no non-trivial $1$-dimensional representations.

Now since the connection $\adel_{\mathcal{E}_k}$ is isomorphic to the $k$-fold tensor product of the connection~$\adel_{\mathcal{E}_1}$, it follows from \eqref{eqn:commdiag} that we have the commutative diagram
$$
\xymatrix{
\Phi(\mathcal{E}_k) \otimes \Phi(\mathcal{E}_l) \otimes \Phi\big(\Omega^{(0,1)}_q(G/L_S)\big) \ar[rrr]^{\Phi(\sigma_k \otimes \mathrm{id}) \, \circ \, \Phi(\mathrm{id} \otimes \sigma_l)} \ar[d]_{\simeq \, \otimes \mathrm{id}} & & & \Phi\big(\Omega^{(0,1)}_q(G/L_S)\big) \otimes \Phi(\mathcal{E}_{k}) \otimes \Phi(\mathcal{E}_{l}) \ar[d]^{\mathrm{id} \otimes \, \simeq}\\
\Phi(\mathcal{E}_{k+l}) \otimes \Phi\big(\Omega^{(0,1)}_q(G/L_S)\big) \,\, \ar[rrr]_{\Phi(\sigma_{k+l})} & & & \Phi\big(\Omega^{(0,1)}_q(G/L_S)\big) \otimes \Phi(\mathcal{E}_{k+l}).\\
}
$$
We now see, for $e \in \mathcal{E}_{k+l}$, $b \in \mathcal{O}_q(G/L_S)$, that
\begin{align} \label{eqn:theta}
\Phi(\sigma_{k+l})\big([e] \otimes \big[\,\adel b\big]\big) = \theta_k\theta_l \big[\,\adel b\big] \otimes [e].
\end{align}
Recalling next Example \ref{eg:trivialBimodConn}, we see that
\begin{align*}
\Phi(\sigma_1)\big([1] \otimes \big[\,\adel b\big]\big) = \big[\,\adel b\big] \otimes [1].
\end{align*}
It now follows from \ref{eqn:theta} that $\theta_{-1} = \theta_1^{-1}$, which in turn implies that the claimed nonzero scalar $\theta \in \mathbb{C}$ is given by $\theta_1$.
\end{proof}

\begin{eg}In this example we determine the constant $\theta$ for the standard Podle\'s sphere $\mathcal{O}_q\big(S^2\big)$. It follows from \eqref{eqn:exadelbar} that, for $e \in \mathcal{E}_{1}$, the element $\Phi(\sigma_1)\big[e \otimes \adel b\big]$ is equal to
\begin{align*}
\Phi(\sigma_1)\big[e \otimes \adel b\big]={} &\big[\,\adel(ebu_{11}) \otimes u_{22}\big] - \big[\,\adel(eu_{11}) \otimes u_{22}b\big] \\
&{} - q\big(\big[\, \adel(ebu_{21}) \otimes u_{12}\big]- q \big[\,\adel(eu_{21}) \otimes u_{12}b\big]\big).
\end{align*}
Choosing $e = u_{22}$ and $b = u_{12}u_{22}$, we get
\begin{align*}
\Phi(\sigma_1)\big[u_{22} \otimes \adel(u_{12}u_{22})\big]&=\big[\,\adel(u_{22}u_{12}u_{22}u_{11}) \otimes u_{22}\big]\\
&=\big[\,\adel(u_{22}u_{12})u_{22}u_{11} \otimes u_{22}\big]+\big[u_{22}u_{12}\adel(u_{22}u_{11}) \otimes u_{22}\big]\\
&=q^{-1}\big[\,\adel(u_{12}u_{22}) \otimes u_{22}\big].
\end{align*}
From which it follows that $\theta = q^{-1}$, showing in particular that $\theta$ is not equal to $1$. This calculation can be extended to the quantum Grassmannians, and the quantum quadrics using the FRT presentation of the bimodule map $\sigma$ given in Section~\ref{subsection:FRT}.
\end{eg}

\begin{prop} \label{prop:Global.Tak}It holds that
\[
\sigma_{k}\big(e \otimes \adel b\big) = \theta^k e_{(-2)} b_{(1)} S\big(e_{(-1)}\big)S\big(b_{(2)}\big)\adel b_{(3)} \otimes e_{(0)},
\]
for any $b \in \mathcal{O}_q(G/L_S)$, and $e \in \mathcal{E}_k$.
\end{prop}
\begin{proof}
A unit for the equivalence is given by
\begin{gather*}
\unit\colon \ \mathcal{F} \to \Psi \circ \Phi(\mathcal{F}), \qquad f \mapsto f_{(-1)} \otimes [f_{(0)}],
\end{gather*}
for any relative Hopf module $\mathcal{F}$. Thus we have that
\begin{align*}
\sigma_k(e \otimes \adel b) &= \unit^{-1} \circ (\mathrm{id} \otimes \Phi(\sigma_k)) \circ \unit\big(e \otimes \adel b\big)\\
&= \unit^{-1} \circ (\mathrm{id} \otimes \Phi(\sigma_k))\big(e_{(-1)}b_{(1)} \otimes \big[e_{(0)} \otimes \adel b_{(2)}\big]\big)\\
&= \theta^k \unit^{-1} \big(e_{(-1)}b_{(1)} \otimes \big[\,\adel b_{(2)} \otimes e_{(0)}\big]\big).
\end{align*}
Recalling the explicit presentation of the inverse of $\unit$ given in \cite[Corollary~2.7]{MMF2}, we see that
\begin{align*}
\theta^k \unit^{-1} \big(e_{(-1)}b_{(1)} \otimes \big[\,\adel b_{(2)} \otimes e_{(0)}\big]\big) &= \theta^k e_{(-2)}b_{(1)} S\big(b_{(2)}e_{(-1)}\big)\adel b_{(3)} \otimes e_{(0)}\\
&= \theta^k e_{(-2)} b_{(1)} S\big(e_{(-1)}\big)S\big(b_{(2)}\big)\adel b_{(3)} \otimes e_{(0)},
\end{align*}
giving us the claimed identity.
\end{proof}

It is interesting to note that in the commutative case, it follows from the antipode axiom for a Hopf algebra that the formula in Proposition \ref{prop:Global.Tak} reduces to the flip map, as it should.

\appendix

\section[Table of simple roots for the irreducible quantum flag manifolds]{Table of simple roots for the irreducible quantum\\ flag manifolds}\label{appendixA}

In this appendix we recall, for the reader's convenience, an explicit description of the irreducible quantum flag manifolds. In the classical setting the irreducible flag manifolds are those flag manifolds $G/L_S$ for which the space of anti-holomorphic forms $\Omega^{(0,1)}_q(G/L_S)$ is irreducible as a~$G$-module. The corresponding subset of simple roots is always of the form $\Pi\backslash \{\alpha_x\}$, such that~$\alpha_x$ has coefficient $1$ in the highest root of $\mathfrak{g}$. In Table~\ref{table1} $\alpha_x$ is denoted graphically by a~coloured node in the Dynkin diagram of~$\mathfrak{g}$, where the nodes are numbered according to \cite[Section~11.4]{Humph}.

\begin{table}[h!]\centering
\caption{}\label{table1}\vspace{1mm}
\begin{tabular}{|c|c|c|c|}
\hline
& & & \\[-2mm]
Series & $\mathcal{O}_q(G)$ & Crossed node & $\mathcal{O}_q(G/L_S)$ \\
& & & \\[-2mm]
\hline
& & & \\[-2mm]
$A_{n}$ & $\mathcal{O}_q(\mathrm{SU}_{n+1})$ &
\begin{tikzpicture}[scale=.5]
\draw
(0,0) circle [radius=.25]
(8,0) circle [radius=.25]
(2,0) circle [radius=.25]
(6,0) circle [radius=.25] ;

\draw[fill=black]
(4,0) circle [radius=.25] ;

\draw[thick,dotted]
(2.25,0) -- (3.75,0)
(4.25,0) -- (5.75,0);

\draw[thick]
(.25,0) -- (1.75,0)
(6.25,0) -- (7.75,0);
\end{tikzpicture} & $\mathcal{O}_q(\mathrm{Gr}_{n+1,m})$\\

& & & \\[-2mm]

$B_n$ & $\mathcal{O}_q(\mathrm{Spin}_{2n+1})$ &
\begin{tikzpicture}[scale=.5]
\draw
(4,0) circle [radius=.25]
(2,0) circle [radius=.25]
(6,0) circle [radius=.25]
(8,0) circle [radius=.25] ;
\draw[fill=black]
(0,0) circle [radius=.25];

\draw[thick]
(.25,0) -- (1.75,0);

\draw[thick,dotted]
(2.25,0) -- (3.75,0)
(4.25,0) -- (5.75,0);

\draw[thick]
(6.25,-.06) --++ (1.5,0)
(6.25,+.06) --++ (1.5,0);

\draw[thick]
(7,0.15) --++ (-60:.2)
(7,-.15) --++ (60:.2);
\end{tikzpicture} & $\mathcal{O}_q(\mathbf{Q}_{2n+1})$ \\

& & & \\[-2mm]

$C_n$ & $\mathcal{O}_q(\mathrm{Sp}_{n})$ &
\begin{tikzpicture}[scale=.5]
\draw
(0,0) circle [radius=.25]
(2,0) circle [radius=.25]
(4,0) circle [radius=.25]
(6,0) circle [radius=.25] ;
\draw[fill=black]
(8,0) circle [radius=.25];

\draw[thick]
(.25,0) -- (1.75,0);

\draw[thick,dotted]
(2.25,0) -- (3.75,0)
(4.25,0) -- (5.75,0);

\draw[thick]
(6.25,-.06) --++ (1.5,0)
(6.25,+.06) --++ (1.5,0);

\draw[thick]
(7,0) --++ (60:.2)
(7,0) --++ (-60:.2);
\end{tikzpicture} & $\mathcal{O}_q(\mathbf{L}_{n})$ \\

& & & \\[-2mm]

$D_n$ & $\mathcal{O}_q(\mathrm{Spin}_{2n})$ &
\begin{tikzpicture}[scale=.5]

\draw[fill=black]
(0,0) circle [radius=.25] ;

\draw
(2,0) circle [radius=.25]
(4,0) circle [radius=.25]
(6,.5) circle [radius=.25]
(6,-.5) circle [radius=.25];

\draw[thick]
(.25,0) -- (1.75,0)
(4.25,0.1) -- (5.75,.5)
(4.25,-0.1) -- (5.75,-.5);

\draw[thick,dotted]
(2.25,0) -- (3.75,0);
\end{tikzpicture} & $\mathcal{O}_q(\mathbf{Q}_{2n})$ \\

& & & \\[-2mm]

$D_n$ & $\mathcal{O}_q(\mathrm{Spin}_{2n})$ &
\begin{tikzpicture}[scale=.5]
\draw
(0,0) circle [radius=.25]
(2,0) circle [radius=.25]
(4,0) circle [radius=.25] ;

\draw[fill=black]
(6,.5) circle [radius=.25];
\draw
(6,-.5) circle [radius=.25];

\draw[thick]
(.25,0) -- (1.75,0)
(4.25,0.1) -- (5.75,.5)
(4.25,-0.1) -- (5.75,-.5);

\draw[thick,dotted]
(2.25,0) -- (3.75,0);
\end{tikzpicture} & $\mathcal{O}_q(\textbf{S}_{n})$ \\

$E_6$ & $\mathcal{O}_q(E_6)$ & \begin{tikzpicture}[scale=.5]
\draw
(2,0) circle [radius=.25]
(4,0) circle [radius=.25]
(4,1) circle [radius=.25]
(6,0) circle [radius=.25] ;

\draw
(0,0) circle [radius=.25];
\draw[fill=black]
(8,0) circle [radius=.25];

\draw[thick]
(.25,0) -- (1.75,0)
(2.25,0) -- (3.75,0)
(4.25,0) -- (5.75,0)
(6.25,0) -- (7.75,0)
(4,.25) -- (4, .75);
\end{tikzpicture}

& $\mathcal{O}_q(\mathbb{OP}^2)$ \\

& & & \\[-2mm]

$E_7$ & $\mathcal{O}_q(E_7)$ &
\begin{tikzpicture}[scale=.5]
\draw
(0,0) circle [radius=.25]
(2,0) circle [radius=.25]
(4,0) circle [radius=.25]
(4,1) circle [radius=.25]
(6,0) circle [radius=.25]
(8,0) circle [radius=.25];

\draw[fill=black]
(10,0) circle [radius=.25];

\draw[thick]
(.25,0) -- (1.75,0)
(2.25,0) -- (3.75,0)
(4.25,0) -- (5.75,0)
(6.25,0) -- (7.75,0)
(8.25, 0) -- (9.75,0)
(4,.25) -- (4, .75);
\end{tikzpicture} & $\mathcal{O}_q(\textbf{F})$
 \\[1mm]
\hline
\end{tabular}
\end{table}

\subsubsection*{Acknowledgements}
AC was supported by the GA\v{C}R project 20-17488Y and \mbox{RVO: 67985840}. AC and R\'OB are supported by the Charles University PRIMUS grant \emph{Spectral Noncommutative Geometry of Quantum Flag Manifolds} PRIMUS/21/SCI/026.
We would like to thank Henrik Winther for useful discussions. Moreover, we would like to thank the referees for their careful readings of the paper and their helpful suggestions.

\pdfbookmark[1]{References}{ref}
\LastPageEnding

\end{document}